\documentclass{amsart}

\usepackage[utf8]{inputenc}
\usepackage{verbatim}
\usepackage{graphicx}
\usepackage{graphicx,caption2,psfrag,float,color}
\usepackage{amssymb}
\usepackage{amscd}
\usepackage{amsmath}

\newtheorem{theorem}{Theorem}[section]

\newtheorem{lemma}[theorem]{Lemma}

\numberwithin{equation}{subsection}
\newtheorem{definition}[theorem]{Definition}

\renewcommand{\Re}{\mathrm{Re}}

\title{A criterion related to the Riemann Hypothesis}
\author{Helmut Maier and Michael Th. Rassias}
\date{\today}
\address{Department of Mathematics, University of Ulm, Helmholtzstrasse 18, 89081 Ulm, Germany.}
\email{helmut.maier@uni-ulm.de}
\address{Institute of Mathematics, University of Zurich, CH-8057, Zurich, Switzerland
 \& Institute for Advanced Study, Program in Interdisciplinary Studies,
1 Einstein Dr, Princeton, NJ 08540, USA.}
\email{michail.rassias@math.uzh.ch, michailrassias@math.princeton.edu}\thanks{}

\begin{document}

\maketitle
 
\begin{abstract} 
A crucial role in the Nyman-Beurling-B\'aez-Duarte approach to the Riemann Hypothesis is played by the distance  
\[
d_N^2:=\inf_{A_N}\frac{1}{2\pi}\int_{-\infty}^\infty\left|1-\zeta A_N\left(\frac{1}{2}+it\right)\right|^2\frac{dt}{\frac{1}{4}+t^2}\:,
\]
where the infimum is over all Dirichlet polynomials 
$$A_N(s)=\sum_{n=1}^{N}\frac{a_n}{n^s}$$
of length $N$.\\
In this paper we investigate $d_N^2$ under the assumption that the Riemann zeta function has four non-trivial zeros off the critical line. Thus we obtain a criterion for the non validity of the Riemann Hypothesis.
\newline

\end{abstract}
\section{Introduction}
The Nyman-Beurling-B\'aez-Duarte approach to the Riemann hypothesis asserts that the Riemann hypothesis is true, if and only if 
$$\lim_{N\rightarrow \infty} d_N^2=0\:,$$
where
\[
d_N^2:=\inf_{A_N}\frac{1}{2\pi}\int_{-\infty}^\infty\left|1-\zeta A_N\left(\frac{1}{2}+it\right)\right|^2\frac{dt}{\frac{1}{4}+t^2}\tag{1.1}
\]
and the infimum is over all Dirichlet polynomials 
$$A_N(s)=\sum_{n=1}^{N}\frac{a_n}{n^s}$$
of length $N$ (see \cite{bcf}).\\
Burnol \cite{burnol}, improving on work of B\'aez-Duarte, Balazard, Landreau and Saias \cite{baez1}, \cite{baez2} showed that 
$$\liminf_{N\rightarrow \infty} d_N^2 \log N\geq \sum_{Re(\rho)=\frac{1}{2}}\frac{m(\rho)^2}{|\rho|^2}\:,$$
where $m(\rho)$ denotes the multiplicity of the zero $\rho$.\\
This lower bound is believed to be optimal and one expects that
\[
d_N^2\sim \frac{1}{\log N}\sum_{Re(\rho)=\frac{1}{2}}\frac{m(\rho)^2}{|\rho|^2}\:.\tag{*}
\]
Under the Riemann hypothesis one has
\[
\sum_{Re(\rho)=\frac{1}{2}}\frac{m(\rho)}{|\rho|^2}=2+\gamma-\log 4\pi\:,\tag{1.2}
\]
where $\gamma$ is the Euler-Mascheroni constant.\\
S. Bettin, J. B. Conrey and D. W. Farmer \cite{bcf} prove (*) under an additional assumption and also identify the Dirichlet polynomials $A_N$, for which the expected infimum in (1.1) is assumed. They prove (Theorem 1 of \cite{bcf}):\\
\textit{Let 
$$V_N(s):=\sum_{n=1}^N\left(1-\frac{\log n}{\log N}\right)\frac{\mu(n)}{n^s}\:.$$
If the Riemann hypothesis is true and if
$$\sum_{|Im(\rho)|\leq T}\frac{1}{|\zeta'(\rho)|^2}\ll T^{\frac{3}{2}-\delta}$$
for some $\delta>0$, then
\[
\frac{1}{2\pi}\int_{-\infty}^\infty\left|1-\zeta V_N\left(\frac{1}{2}+it\right)\right|^2\frac{dt}{\frac{1}{4}+t^2}\sim\frac{2+\gamma-\log 4\pi}{\log N}\:.\tag{1.3}
\]}
In this paper we investigate the expression (1.3) under an assumption contrary to the Riemann hypothesis: There  are exactly four nontrivial zeros off the critical line. We observe that nontrivial zeros off the critical line always appear as 
quadruplets. Indeed, if $\zeta(\rho)=0$ for $\rho=\sigma+i\gamma$ with 
$1>\sigma>\frac{1}{2}$, $\gamma>0$, then from the functional equation
\[
\Lambda(s)=\Lambda(1-s)\:,\tag{1.4}
\]
where
$$\Lambda(s):=\pi^{-s/2}\Gamma\left(\frac{s}{2}\right)\zeta(s)$$
and the trivial relation $\zeta(\bar{s})=\overline{\zeta(s)}$, we obtain that 
$$\zeta(\sigma+i\gamma)=\zeta(1-\sigma+i\gamma)=\zeta(\sigma-i\gamma)=\zeta(1-\sigma-i\gamma)=0\:.$$
We prove the following:
\begin{theorem}\label{thm1}
Let $\sigma_0>1/2$, $\gamma_0>0$, $\zeta(\sigma_0\pm i\gamma_0)=\zeta(1-\sigma_0\pm i\gamma_0)=0$ and $\zeta(\sigma+i\gamma)\neq 0$ for all other
$\sigma+i\gamma$ with $\sigma>1/2$. Assume that
$$\sum_{|Im(\rho)|\leq T}\frac{1}{|\zeta'(\rho)|^2}\ll T^{\frac{3}{2}-\delta}\ \ (T\rightarrow\infty),$$
for some $\delta>0$. Then, there are constants $A=A(\sigma_0, \gamma_0)$ and
$B=B(\sigma_0, \gamma_0)\in\mathbb{R}$, such that for all $\epsilon>0$:
$$\frac{1}{2\pi}\int_{-\infty}^\infty\left|1-\zeta V_N\left(\frac{1}{2}+it\right)\right|^2\frac{dt}{\frac{1}{4}+t^2}$$
$$=\frac{1}{\log^2 N}\left(A N^{2\sigma_0-1}\cos(2\gamma_0\log N)+BN^{2\sigma_0-1}\right)\left(1+O(N^{\frac{1}{2}-\sigma_0+\epsilon})\right),$$
for some $\epsilon>0$.
\end{theorem}
\section{Preliminary Lemmas and Definitions}
\begin{lemma}\label{lem21}
Let $\epsilon>0$ be fixed but arbitrarily small. Under the assumptions of Theorem
\ref{thm1} we have 
\[
\zeta(\sigma+i t)\ll |t|^{3\epsilon}\ \ (|t|\rightarrow \infty)\:,\tag{2.1}
\]
for 
$$\frac{1}{2}-\epsilon\leq \sigma\leq \frac{1}{2}+\epsilon\:.$$

\end{lemma}
\begin{proof}
The estimate (2.1) is well known as the Lindel\"of hypothesis, which is a consequence of the Riemann hypothesis. In \cite{Titchmarsh}, the 
Lindel\"of hypothesis is proven on the assumption of the Riemann hypothesis.
This proof may be adapted to the new situation by slight modifications.\\
The function $\log \zeta(s)$ is holomorphic in the domain
$$\mathcal{G}:=\left\{\sigma+it\::\: \sigma>\frac{1}{2}\right\}\setminus\left\{\left[\frac{1}{2},1\right]\cup\left[\frac{1}{2}+i\gamma_0, \sigma_0+i\gamma_0\right]\cup \left[\frac{1}{2}-i\gamma_0, \sigma_0-i\gamma_0\right]\right\}\:.$$
Let now $\frac{1}{2}<\sigma^*\leq \sigma\leq 1$. As in \cite{Titchmarsh}, let $z=\sigma+it$, but
now $|t|$ sufficiently large.\\
We apply the Borel-Carath\'eodory theorem to the function $\log\zeta(z)$ and
the circles with centre $2+it$ and radius $\frac{3}{2}-\frac{1}{2}\delta$ and
$\frac{3}{2}-\delta$, ($0<\delta<\frac{1}{2}$).\\
On the larger circle
$$Re(\log \zeta(z))=\log |\zeta(z)|<A\log t$$
for a fixed positive constant $A$. Hence on the smaller circle 
$$|\log\zeta(z)|\leq \frac{3-2\delta}{\frac{1}{2}\delta}A\log t+\frac{3-\frac{3}{2}\delta}{\frac{1}{2}\delta}|\log|\zeta(2+it)||<A\delta^{-1}\log t\:.$$
We now apply Hadamard's three circle theorem as in \cite{Titchmarsh}. The 
proof there can be taken over without change, to obtain
\[
\zeta(z)=O(t^\epsilon),\ \text{for every $\sigma>\frac{1}{2}\:,$}\tag{2.2}
\]
which is (14.2.5) of \cite{Titchmarsh}.\\
By the functional equation (1.4) we obtain
\[
\left|\zeta\left(\frac{1}{2}-\epsilon+it\right)\right|=O\left(|t|^{3\epsilon}\right)\:.\tag{2.3}
\]
The claim (2.1) now follows from (2.2), (2.3) and the theorem of 
Phragm\'en-Lindel\"of.
\end{proof}
\begin{definition}\label{def22}
For $\rho$ a non-trivial zero of $\zeta(s)$ let 
$$R_N(\rho, s):=Res_{z=\rho}\frac{N^{z-s}}{\zeta(z)(z-s)^2}$$
and
$$F_s(z):=\pi z^s\sum_{n=1}^\infty(-1)^n\frac{(2\pi)^{2n+1}z^{2n}}{(2n)!\:\zeta(2n+1)(2n+s)^2}$$
\end{definition}
\begin{lemma}\label{lem23}
If $0<Re(s)<1$, then
$$V_N(s)=\frac{1}{\zeta(s)}\left(1-\frac{1}{\log N}\frac{\zeta'}{\zeta}(s)\right)+\frac{1}{\log N}\sum_{\rho}R_N(\rho, s)+\frac{1}{\log N}\:F_s\left(\frac{1}{N}\right)\:,$$
where the sum is over all distinct non-trivial zeros of $\zeta(s)$.
\end{lemma}
\begin{proof}
This is Lemma 2 of \cite{bcf}.
\end{proof}
\begin{lemma}\label{lem24}
Let $\epsilon >0$. Under the assumptions of Theorem \ref{thm1} we have 
$$\sum_{\rho, Re(\rho)=\frac{1}{2}}R_N(\rho, s)\ll N^{\mp\epsilon}|s|^{\frac{3}{4}-\frac{\delta}{2}+\epsilon}\:.$$
\end{lemma}
\begin{proof}
The proof is identical to the proof of Lemma 3 of \cite{bcf}. There the summation
condition $Re(s)=1/2$ is not needed, since the Riemann hypothesis is assumed.
\end{proof}
\begin{lemma}\label{lem25}
$$N^{\pm\epsilon}\sum_{\rho, |\rho|=\frac{1}{2}}R_N(\rho, s)\ll\sum_{\substack{|\rho-s|<\frac{|\rho|}{2}\\Re(\rho)=\frac{1}{2} }}\frac{1}{|\zeta'(\rho)||\rho-s|^2}+1\:.$$
\end{lemma}
\begin{proof}
This is (5) of \cite{bcf}.
\end{proof}
\begin{definition}\label{def26}
We set
$$\Sigma^{(1)}(N,s):=\frac{1}{\log N}\sum_{\rho:Re(\rho)=\frac{1}{2}}R_N(\rho, s)$$
and
$$\Sigma^{(2)}(N,s):=\frac{1}{\log N}\sum_{\rho\in\{\sigma_0\pm i\gamma_0\}}R_N(\rho, s)$$
\end{definition}
\section{Proof of Theorem \ref{thm1}}
We closely follow the proof of Theorem 1 in \cite{bcf}. We have
\begin{align*}
\frac{1}{2\pi}\int_{-\infty}^\infty\left|1-\zeta V_N\left(\frac{1}{2}+it\right)\right|^2\frac{dt}{\frac{1}{4}+t^2}&=\frac{1}{2\pi i}\int_{(\frac{1}{2})}(1-\zeta V_N(s))(1-\zeta V_N(1-s))\frac{ds}{s(1-s)}\\
&=\frac{1}{2\pi i}\int_{(\frac{1}{2}-\epsilon)}(1-\zeta V_N(s))(1-\zeta V_N(1-s))\frac{ds}{s(1-s)}
\end{align*}
By Lemma \ref{lem23} and Definition \ref{def26} this is 
\begin{align*}
&\ \ \frac{1}{\log^2 N}\:\frac{1}{2\pi i}\int_{(\frac{1}{2}-\epsilon)}\left(\frac{\zeta'}{\zeta^2}(s)-\Sigma^{(1)}(N,s)-\Sigma^{(2)}(N,s)-F_s\left(\frac{1}{N}\right)\right)\times\tag{3.1}\\
&\left(\frac{\zeta'}{\zeta^2}(1-s)-\Sigma^{(1)}(N,1-s)-\Sigma^{(2)}(N,1-s)-F_{1-s}\left(\frac{1}{N}\right)\right) \frac{\zeta(s)\zeta(1-s)}{s(1-s)}\:ds\:.
\end{align*}
We now expand the product in (3.1) and separately estimate the products that do not contain terms $\Sigma^{(2)}$ and the products consisting of a term
$\Sigma^{(2)}$ and another term.\\
The asymptotic finally is obtained by asymptotically evaluating the products consisting only of factors $\Sigma^{(2)}$.\\
We closely follow \cite{bcf}. It follows from Lemmas \ref{lem21}, \ref{lem23},
\ref{lem24} that 
\begin{align*}
&\ \ \frac{1}{\log^2 N}\:\frac{1}{2\pi i}\int_{(\frac{1}{2}-\epsilon)}\sum_{\rho_1, \rho_2} R_N(\rho_1, s)R_N(\rho_2, 1-s)\frac{\zeta(s)\zeta(1-s)}{s(1-s)}\: ds\\
&\ll\frac{1}{\log^2 N}\int_{(\frac{1}{2}-\epsilon)}\sum_{|\rho-s|<\frac{|\rho|}{2}}\frac{1}{|\zeta'(\rho)||\rho-s|^2}\frac{|ds|}{|s|^{\frac{5}{4}+\frac{\delta}{2}-5\epsilon}}+O\left(\frac{1}{\log^2 N}\right)\:.\\
\end{align*}
Now by Lemma \ref{lem24} and the trivial estimate 
$$F_s\left(\frac{1}{N}\right)=O(N^{-5/2})\:,$$
all the other terms in (3.1) not containing factors $\Sigma^{(2)}$ are trivially 
$O(1/\log^2 N)$ apart from 
\begin{align*}
-\frac{1}{\log^2 N}\:\frac{1}{2\pi i}\int_{(\frac{1}{2}-\epsilon)}\frac{\zeta'}{\zeta}(1-s)\Sigma^{(1)}(N,s)\frac{\zeta(s)\zeta(1-s)}{s(1-s)}\: ds&=\log N-\frac{1}{2}\frac{\zeta''}{\zeta'}(\rho)+\frac{\chi'}{\chi}(\rho)+\frac{1-2\rho}{|\rho|^2}\\
&=\frac{\log N}{|\rho|^2}+O\left(\frac{1}{|\rho|^{2-\epsilon}|\zeta'(\rho)|}+\frac{1}{|\rho|^2} \right)\:,
\end{align*}
where we set 
$$\chi(s):=\pi^{-s/2}\Gamma\left(\frac{s}{2}\right)$$ 
and use the bound 
$$\zeta''\left(\frac{1}{2}+it\right)\ll |t|^\epsilon\:,$$
which follows from Lemma \ref{lem21} by the well-known estimate for the derivatives of a holomorphic function. By moving the line of integration to 
$Re(s)=\frac{1}{2}+\epsilon$, we get that the contribution from the products
not containing $\Sigma^{(2)}$ is
$$\frac{1}{\log N}\sum_{\rho:Re(\rho)=\frac{1}{2}}\frac{1}{|\rho|^2}+O\left(\frac{1}{\log^2 N}\right)\:.$$
We now come to the products that contain the factor $\Sigma^{(2)}(N,s)$.\\
They may be handled by adding the factor $N^{\sigma_0-\frac{1}{2}+\epsilon}$
stemming from $N^{z-s}$ in Definition \ref{def22}.\\
These estimates yield the error-term in Theorem \ref{thm1}. The main term finally is obtained by evaluating the contribution with two factors $\Sigma^{(2)}$
and by observing that the integral in (1.3) is real. \\
\qed

\end{document}